\documentclass[12pt]{amsart}

\usepackage{amsfonts, amssymb, amsmath, amsthm,epsf,graphicx}
\newcommand{\R}{\mathbb{R}}

\newcommand{\Q}{\mathbb{Q}}
\newcommand{\Z}{\mathbb{Z}}

\newcommand{\bT}{\mathbf{T}}

\newcommand{\Lim}{\varprojlim} 

\newcommand{\larr}{\left( \begin{array}{c}}
\newcommand{\rarr}{\end{array} \right) }

\newcommand{\lsqarr}{\left[ \begin{array}{c}}
\newcommand{\rsqarr}{\end{array} \right]}

\newtheorem{theorem}{Theorem}

\newtheorem{definition}[theorem]{Definition}

\newtheorem{lemma}[theorem]{Lemma}

\begin{document}

\title{Exact Regularity and the Cohomology of Tiling Spaces}
\author{Lorenzo~Sadun}


\begin{abstract} Exact regularity was introduced recently as a
  property of homological Pisot substitutions in one dimension. In
  this paper, we consider the analog of exact regularity for arbitrary
  tiling spaces. Let $\bT$ be a $d$ dimensional repetitive tiling, and
  let $\Omega$ be its hull. If $\check H^d(\Omega, \Q) =
  \Q^k$, then there exist $k$ patches whose appearance govern the
  number of appearances of every other patch. This gives uniform
  estimates on the convergence of all patch frequencies to the ergodic
  limit. If the tiling $\bT$ comes from a substitution, then we can
  quantify that convergence rate. If $\bT$ is also one-dimensional, we
  put constraints on the measure of any cylinder set in
  $\Omega$.
\end{abstract}

\maketitle

\section{Introduction and Statement of Results}\label{sec:intro}

Ever since the seminal paper of Anderson and Putnam \cite{AP}, there 
has been a small industry devoted to computing topological invariants
of tiling spaces. A key question throughout this effort has been ``what 
do these invariants actually mean?'' Put another way, if we determine (say)
that the first cohomology of a 1-dimensional tiling space is $\Z[1/2] \oplus
\Z$, what does that tell us about the properties of tilings in that 
space? Progress has been made, relating tiling cohomology to gap labeling
\cite{BBG}, to deformations of tilings \cite{CS}, and to measures and 
patch frequencies \cite{Jean-Marc}. In this paper we continue this last
direction of inquiry and show how the top-dimensional cohomology governs
not only patch frequencies over an entire tiling, but the number of 
appearances of a patch in any finite region. 

Exact regularity was introduced
in \cite{HPC}. In that paper, we considered one-dimensional
substitution tilings with tile lengths chosen according to the left
eigenvector of the substitution matrix, where the stretching factor $\lambda$
is a Pisot number of algebraic degree $k$. 
If the rank of the first rational (\v Cech) cohomology
of the tiling space is also $k$, then the number of appearances of
a patch $P$ in a return word is determined {\em exactly} by the Euclidean length
(in $\R^1$) of the return word.  

The proof never used the condition that $\lambda$ was a Pisot
number. It did rely on the tiling being 1-dimensional, and on the
dimension of $\check H^1$ equaling the algebraic degree of the
stretching factor.  The point is that, under these conditions, the
integer span of the tile lengths is a rank-$k$ free Abelian group, so
that specifying the length (in $\R$) of a return word is equivalent to
specifying $k$ integers. These $k$ integers
then determine how many times $P$ appears, up to a
boundary term. For appropriately chosen return words, the contributions of 
the two boundaries cancel and we are left with an exact formula. Tilings for
which this exact formula 
works are said to have the {\em Exact Regularity Property}, or ERP.

In this paper we study the situation where the dimension
of $\check H^1$ is unrelated to the algebraic degree of the stretching 
factor, or where the tiling does not come from a substitution at all, or 
when the tiling is of a higher-dimensional Euclidean space. 
In each of these cases the rank
of the top cohomology determines how many integers 
are needed to keep track of an arbitrary patch $P$ in an arbitrary region $R$
(Theorem \ref{general}).  
The formulas are generally 
{\em not} exact, but involve correction terms that depend
on the boundary of $R$. However, this is sufficient
to prove estimates on the rate at which the frequency of $P$ approaches 
its ergodic average (Theorems \ref{ergodic} and \ref{sub-rate}).

In some cases (e.g., substitution
tilings in one dimension, and some special 2-dimensional examples) it is 
possible to choose a region where the boundary terms vanish. This allows 
us to set constraints on the possible measures of cylinder sets in 
1-dimensional substitution tilings (Theorem \ref{cylinder}).

Let $\bT_0$ be a tiling of $\R^d$ that is translationally finite and
repetitive. Being translationally finite (also called having finite
local complexity) means that there are only a finite number of tile
types, up to translation and there are only a finite number of ways
that two tiles can meet. Being repetitive means that, for every patch
$P$, there an 
$r_P$ such that in every ball of radius $r_P$ there is at least
one copy of $P$. Note that these assumptions exclude the pinwheel
tiling, where tiles point in an infinite number of directions, and all
tilings where patterns appear with frequency zero (e.g., a one
dimensional tiling with one black tile and infinitely many white
tiles). It is possible to address the pinwheel tiling, but the techniques
are somewhat different. See the end of Section \ref{examples}.

The {\em hull} $\Omega$ of $\bT_0$, also called the tiling space
associated with $\bT_0$, is the set of tilings $\bT$ with the property
that every patch in $\bT$ is a translate of a patch in $\bT_0$. Under
the above assumptions, $\Omega$ is a minimal dynamical system with
respect to translations and is a compact space, with a metric where
two tilings are close if they agree on a big ball centered at the
origin, up to a small translation. If $\bT$ is in the hull of $\bT_0$,
then the hull of $\bT$ is the same as the hull of $\bT_0$. For this
reason, we usually speak of a minimal tiling space $\Omega$, rather
than the hull of any one particular tiling.

\begin{theorem}\label{general} 
Let $\Omega$ be a compact and minimal space of translationally finite
tilings of $\R^d$.
Suppose that $\check H^d(\Omega, \Q)
= \Q^k$ for some integer $k$. Then there exist patches $P_1, \ldots, P_k$
with the following property:  For any other patch $P$ 
there exist rational numbers $c_1(P), \ldots, c_k(P)$ such that, for any
region $R$ in any tiling $\bT \in \Omega$, 
the number of appearances of $P$ in $R$ equals 
$\sum_{i=1}^k c_i(P) n_i + e(P,R)$, where $n_i$ is the number of appearances
of $P_i$ in $R$, and $e(P,R)$ is an error term computable from the patterns
that appear 
on the boundary of $R$. In particular, the magnitude of $e(P,R)$ is bounded
by a constant times the measure of the boundary of $R$. 
\end{theorem}

Note that Theorem \ref{general} is not limited to substitution
tilings, nor even to non-periodic tilings, but applies to any
compact and minimal tiling space. For instance, if $\bT_0$
is a periodic tiling, then $\Omega$ is a $d$-dimensional torus and 
$\check H^d(\Omega, \Q)=\Q$, so we only need to count one patch 
(say, a fundamental domain) to determine the number of all other 
patches. The Penrose tiling
has $\check H^2(\Omega, \Q)=\Q^8$, so there are eight patches whose 
appearance controls the appearance of all other patches. Further examples
are given in Section \ref{examples}.

This theorem has immediate
implications for the existence of ergodic limits, and for the rate of
convergence to those limits.

\begin{theorem}\label{ergodic} Suppose that a tiling space $\Omega$ 
satisfies the conditions of Theorem \ref{general}, and suppose that the
patches $P_1, \ldots, P_k$ occur with well-defined frequencies 
$f_1, \ldots, f_k$. That is, the number of occurences of $P_i$ in any 
ball of volume $V$, divided by $V$, approaches $f_i$ as $V \to \infty$. 
Then $\Omega$ is uniquely ergodic, and the frequency of any patch
$P$ approaches $\sum c_i(P)f_i$ at least as quickly as the slowest of the
$P_i$'s, or as $V^{-1/d}$, whichever is slower. 
\end{theorem}

In particular, a uniquely ergodic tiling space whose patch frequencies do
not converge uniformly must have infinitely generated rational top cohomology.

This theorem is in some sense dual to the results of \cite{Jean-Marc},
who study possible measures on tilings spaces by considering frequencies
of (possibly collared) tiles and requiring that they satisfy a set of
homological ``Kirchoff's Rules''. Knowing the (co)homology of the 
tiling space then puts constraints on the possible invariant measures.

When the tiling comes from a substitution, we can further 
sharpen our convergence estimates. 
Let $\bT$ be a repetitive and non-periodic tiling derived from a
primitive self-similar substitution.
Let $M$ be the matrix of the substitution. That is,
$M_{ij}$ is the number of times that the $i$-th tile type appears in the
substitution of the $j$-th tile type. Arrange the eigenvalues $\lambda_i$ 
of $M$ in decreasing order of size. 
Note that $\lambda_1$ is real and positive
and is strictly larger than $|\lambda_2|$. 

\begin{theorem}\label{sub-rate} If $\bT$ is a self-similar tiling as above, 
and if $P$ is any patch in $\bT$,
then there exist constants $K$ and $\nu$ such that, for any ball $R$ 
of sufficiently large volume $V$, 
\begin{equation} \left | \frac{\hbox{number of $P$'s in $R$}}{V}
- \sum c_i(P) f_i \right | < K V^{-\gamma}(\log V)^\nu,
\end{equation}
where $\gamma = \min(d^{-1}, 1 - \frac{\log |\lambda_2|}{ \log(\lambda_1)})$.
\end{theorem}

We obtain our strongest results when the tiling is 1-dimensional and comes 
from a substitution. 
Let $\Omega$ be a 1-dimensional tiling space coming
from a primitive 
substitution $\phi$. The top cohomology  $\check H^1(\Omega, \Q)=\Q^k$ 
is finitely generated \cite{BD}. 
The substitution $\phi$ maps $\Omega$ to itself, and therefore
maps $\check H^1(\Omega, \Q)$ to itself. This last action can be expressed by 
a non-singular $k \times k$ integer matrix $A$. 
(Note: $A$ is typically different from 
the matrix $M$ of the substitution, but both matrices have the same 
leading eigenvalue \cite{BD}, namely the stretching factor $\lambda$.)

The {\em minimal polynomial} of a matrix $A$ is the lowest order monic
polynomial $p(x)$ for which $p(A)=0$. Likewise, the minimal polynomial
of the leading eigenvalue $\lambda$ is the lowest order monic
polynomial $q(x)$ for which $q(\lambda)=0$.  Since $p(\lambda)=0$,
$q(x)$ divides $p(x)$ and we can write $p(x) = q(x) r(x)$ for some
integer polynomial $r(x)$. If $A$ is primitive, then the polynomials
$q(x)$ and $r(x)$ have no roots in common, so we can find a nonzero
integer $D$ and polynomials $Q(x)$ and $R(x)$ with integer
coefficients such that $Q(x)q(x) + R(x) r(x) = D$.  The smallest such
integer $D$ is called the {\em congruence number} \cite{Weise-Taixes}
or {\em reduced resultant \cite{Pohst}} of $q(x)$ and $r(x)$.

Connected patches in a 1-dimensional tiling correspond to sequences of letters,
also called words, where each letter designates a tile type. If the
word $\ell_1 \ldots \ell_n \ell_{n+1}$ occurs in the tiling,
and if $\ell_1=\ell_{n+1}$, then $\ell_1 \ldots \ell_n$ is called a {\em 
return word} and the total length of the tiles corresponding to
the letters $\ell_1 \ldots \ell_n$ is called a {\em return length}. 

\begin{theorem}\label{cylinder}
Let $\Omega_\phi$ be a tiling space obtained from a primitive
one-dimensional substitution 
$\phi$, let $\lambda$ be the stretching factor of $\phi$, let 
$L>0$ be a return length, let $\bT \in \Omega_\phi$, and let
$P$ be any patch in $\bT$. Then the frequency of $P$ in $\bT$
takes the form
\begin{equation}
f(P) = \frac{u_P(\lambda)}{LD q'(\lambda) q_0^n},
\end{equation} 
where $u_P(x)$ is a polynomial with integer coefficients, $q(x)$ is
the minimal polynomial of $\lambda$, $q'(x)$ is its derivative, 
$q_0$ is the constant coefficient of 
$q(x)$, and $D$ is the reduced resultant of $q(x)$ and $r(x)$. 
\end{theorem}

This theorem has a very different flavor from Theorems
\ref{general}--\ref{sub-rate}. The previous theorems are ergodic in
nature and describe convergence to an ergodic limit, while Theorem 
\ref{cylinder} is algebraic and puts constraints on what that limit can be. 

In Section \ref{background} we
provide necessary background for understanding the proofs of these
theorems. In Section \ref{proofs} we prove the four theorems, and in
Section \ref{examples} we provide examples that illustrate how exact
regularity works in practice.

\section{Background}\label{background}

A tile is a topological disk that is the closure of its interior,
together with a label. A tiling of $\R^d$ is a collection of tiles
that overlap only on their boundaries, whose union is all of $\R^d$.
A {\em patch} is a finite collection of tiles that intersect only on
their boundaries, and the support of a patch is the union of these tiles. 
$[{B}_r]^\bT$ denotes the patch consisting of all tiles that 
intersect a closed ball of
radius $r$ around the origin in the tiling $\bT$. The translation group
acts naturally on tilings by moving each tile simultaneously. If $\bT$
is a tiling, then $\bT-x$ is the same tiling translated by the vector
$x$, so that a neighborhood of the origin in $\bT-x$ looks like a
neighborhood of $x$ in $\bT$. 

A tiling {\em space} $\Omega$ 
is a topological space whose elements are tilings, 
with a metric topology in which two tilings are close if they agree on
a big ball around the origin, up to a small translation. In this 
paper we assume throughout 
that $\Omega$ is compact and is a minimal dynamical system
with respect to translations. This is equivalent to $\Omega$ being the
hull of a tiling that is translationally finite (for compactness) and 
repetitive (for minimality). 

Every compact tiling space is topologically conjugate to another tiling space
whose tiles are polytopes. Without loss of generality, we henceforth assume
that all of our tiles are polytopes. 

A self-similar substitution is a map that replaces each tile with a patch
whose support is the original tile, scaled by a {\em linear stretching factor}
$\lambda_0$. This action extends by concatenation to patches, and indeed
to tilings. A {\em supertile of order $n$} is a patch obtained by applying
the substitution $n$ times to a tile. 

The {\em matrix of the substitution} (also called the {\em abelianization} when
the tilings in question are 1-dimensional) is the matrix $M$ whose entries
$M_{ij}$ count the number of times that the $i$-th tile type appears in a 
substituted $j$ tile.  
A square matrix $A$ with non-negative entries 
is {\em primitive} if all the entries of
some power $A^n$ are positive. Every primitive matrix has a largest
positive eigenvalue $\lambda_{PF}$, called the Perron-Frobenius
eigenvalue. This eigenvalue has (algebraic and geometric) 
multiplicity one, and the
corresponding left- and right-eigenvectors have strictly positive
entries. All other eigenvalues are strictly smaller than $\lambda_{PF}$ in
magnitude. 

If $\phi$ is a tiling substitution with matrix $M$, and if $M$ is
primitive, then the relative frequencies of all tile types is given by
the right Perron-Frobenius eigenvector of $M$. The relative volumes of
the tile types in a self-similar tiling are given by the left
Perron-Frobenius eigenvector of $M$. The Perron-Frobenius eigenvalue
is $\lambda=\lambda_0^d$, where $\lambda_0$ is the linear stretching
factor.

A substitution $\phi$ {\em forces the border} \cite{K1} if there is an
integer $n$ with the following property. If $t_1$ and $t_2$ are tiles
in $\bT$ of the same type, then the patches $\phi^n(t_1)$ and
$\phi^n(t_2)$ not only agree in $\phi^n(\bT)$, up to translation, but their
nearest neighbors also agree. If a substitution forces the border,
then the tiling space is the inverse limit of the Anderson-Putnam
complex \cite{AP}, and all cohomology classes are generated by the
duals to the vertices, edges, faces, etc., of this complex. If a
substitution does not force the border, then we can define a new tile
set using {\em collared tiles}. A collared tile is a tile
together with a label describing the immediate neighbors of that tile.
For instance, in the tiling $\ldots baab \ldots$, the two $a$'s are different
as collared tiles, in that the first is preceded by a $b$ and followed by
an $a$, while the second is preceded by an $a$ and followed by a $b$. 

Collaring greatly increases the number of tile types, leading to a more
complicated substitution matrix. The new substitution
forces the border while describing the same tiling space as the original
\cite{AP}.

Let $M$ be the matrix of a self-similar substitution and 
let $M'$ be the matrix of the same substitution rewritten in
terms of collared tiles. If $t_1$ and $t_2$ are tiles of the same type, but
are different as collared tiles, then the supertiles 
$\phi^n(t_1)$ and $\phi^n(t_2)$
are the same (up to translation) as collections of regular tiles. 
Viewed as collections of collared tiles, they can disagree
only near the boundary. Thus the differences between substituted collared tiles
of the same uncollared type can grow at most as $\lambda_0^{n(d-1)}$.
This implies that, if an eigenvalue $\lambda'$ of $M'$ is not an
eigenvalue of $M$, then $|\lambda'| \le \lambda_0^{d-1}$. 

A 1-dimensional substitution is {\em proper} if every substituted
letter begins with the same letter, and every substituted letter ends with
the same letter.  For instance, the substitution $a \to abbabb$, $b \to aabab$
is proper, in that all substituted letters start with $a$ and end with
$b$. In a proper substitution, every substituted letter is a return
word. It is always possible to rewrite a substitution to make it
proper, without changing the underlying tiling space. This rewriting
will change the substitution matrix, but the only eigenvalues that can
appear or disappear from the rewriting process are zero and roots of
unity. In particular, the Perron-Frobenius eigenvalue is unchanged. If
a substitution is proper, then the first \v Cech cohomology is 
the direct limit of the transpose of the substitution matrix \cite{BD}.

Pattern-equivariant cohomology was first defined by Kellendonk and Putnam
\cite{Kel, KP} using differential forms, and then recast in \cite{integer}
in terms of cochains. Here we consider rational pattern-equivariant 
cohomology using cochains. 

A $d$-dimensional tiling $\bT$ gives $\R^d$
the structure of a CW complex, with the vertices serving as
0-cells, the edges serving as $1$-cells, the 2-dimensional faces as 
$2$-cells, and so on. We consider rational cellular 
cochains on this CW complex, 
with a $n$-cochain assigning a rational number to each $n$-cell.
\begin{definition}
A rational $0$-cochain is said to be \emph{pattern-equivariant with
    radius $r$} if, whenever $x$ and $y$ are vertices of $\bT$ and
$[{B}_r]^{\bT-x}=[{B}_r]^{\bT-y}$,
  the cochain takes the same values at $x$ and $y$. A \ $0$-cochain is
  \emph{pattern-equivariant} if it is pattern-equivariant with radius
  $r$ for some finite $r$. Pattern-equivariant $n$-cochains for $n>0$ are
  defined similarly -- their values on a $n$-cell depend only on the
  pattern of the tiling out to a fixed finite distance around that
  $n$-cell.
\end{definition}
If $\beta$ is a rational pattern-equivariant $n$-cochain, its
coboundary, $\delta_n(\beta)$, is a rational pattern-equivariant
$(n+1)$-cochain.
\begin{definition}
The rational $n$-th pattern-equivariant
cohomology of $\bT$ is $H^n_{PE}(\bT,\Q)=Ker(\delta_n)/Im(\delta_{n-1})$. 
\end{definition}
A priori this would seem to depend on $\bT$, but if the tiling
space $\Omega$ is minimal, then
$H^n(\bT,\Q)$ is the same for all 
$\bT \in \Omega$ and is isomorphic to $\check
H^n(\Omega, \Q)$ \cite{Kel, KP, integer}. (Even if $\Omega$ is not minimal,
$H_{PE}(\bT,\Q)$ is isomorphic to $\check H^n(\Omega_\bT, \Q)$, 
where $\Omega_\bT$ is the hull of $\bT$.) 

An {\em indicator cochain} for a patch $P$ is a $d$-cochain that evaluates to
1 on a particular tile of $P$, and evaluates to 0 on all other tiles, whether
in $P$ or not. In other words, it 
counts the occurrences of $P$. It's easy to see that all indicator cochains
for a specific $P$ are cohomologous, and that every pattern-equivariant
$d$-cochain is a linear combination of indicator cochains. 
In particular, every cohomology class in $H^d_{PE}(\bT,\Q)$ can be
represented by a linear combination of indicator cochains. 

If two monic polynomials $q(x)$ and $r(x)$ with integer coefficients
have no roots in common, then the {\em resultant} of $q$ and $r$ is
$Res(q,r)= \prod_{i,j}(\lambda_i - \mu_j)=\prod_{i}r(\lambda_i) = \pm
\prod_j q(\mu_j)$, where $q(x) = \prod (x-\lambda_i)$ and
$r(x)=\prod(x-\mu_j)$. This quantity is easily computed and is closely
related to the reduced resultant $D$ discussed earlier. $D$ and
$Res(q,r)$ have the same prime factors and $D$ always divides
$Res(q,r)$, but the two numbers are not always equal. Computing $D$ is
usually more difficult; see \cite{Weise-Taixes} for an algorithm.

\section{Proofs}\label{proofs}

\begin{proof}[Proof of Theorem \ref{general}]
Every pattern-equivariant cochain is a linear combination of indicator
cochains. Furthermore, each $d$-dimensional cochain is closed, and hence
defines a cohomology class. If $\check H^d(\Omega, \Q)$ is $k$-dimensional, 
we can find $k$ patches $P_1, \ldots, P_k$, whose indicator cochains
$\chi_1, \ldots, \chi_k$ represent linearly independent classes in 
$\check H^d(\Omega, \Q)$. Let $[\chi_i]$ be the cohomology class of 
$\chi_i$. 
Likewise, let $\chi_P$ be an indicator
cochain of the patch $P$, and let $[\chi_P]$ be its cohomology class. 
Since $\{[\chi_1], \ldots, [\chi_k]\}$ is a basis for 
$\check H^d(\Omega, \Q)$, 
there exist rational numbers $c_1(P), \ldots, c_k(P)$ such that
$[\chi_P] = \sum_i c_i(P) [\chi_i]$, hence 
\begin{equation}\label{chiP}
\chi_P = \sum c_i(P) \chi_i + \delta \alpha,
\end{equation}
for some pattern-equivariant $(d-1)$-cochain $\alpha$.  
Now apply both sides of equation (\ref{chiP})
to a region $R$. The left hand side gives the number of $P$'s in
$R$, while the right-hand side gives $\alpha(\partial R) + 
\sum_{i=1}^k c_i(P) n_i$, where $\partial R$ is the boundary of $R$, viewed
as a chain.  The term $\alpha(\partial R)$ is our error term
$e(P,R)$, and is bounded in magnitude 
by a constant times the size of $\partial R$. 
\end{proof}

Note that we have actually proved something stronger than Theorem
\ref{general}, since we have obtained a formula for the error. This
will become important when we consider 1-dimensional tilings, and some
special 2-dimensional tilings, where
for appropriate regions we can get the error term to vanish. 

\begin{proof}[Proof of Theorem \ref{ergodic}]

This is an immediate corollary of Theorem \ref{general}. For any region $R$
of volume $V$, let $n$ be the number of times $P$ appears in $R$, and let
$n_i$ be the number of times that $P_i$ appears. 
\begin{eqnarray*}
\frac{n}{V}- \sum f_i c_i(P) 
&=& \frac{e(P,R) +  \sum n_i c_i(P)}{V} - \sum f_i c_i(P) \cr
&=& \frac{e(P,R)}{V} + \sum c_i(P) (\frac{n_i}{V} - f_i).
\end{eqnarray*}
The first term goes to zero as $V^{-1/d}$, while the others converge 
at worst at the rate of the slowest $P_i$. \end{proof}

Before proving Theorem \ref{sub-rate}, which concerns the 
convergence of the frequencies of arbitrary patches, we consider the 
convergence of the frequencies of the basic tile types.

\begin{lemma}\label{supertiles}
There exist
constants $c_0$ and $\nu_0$ such that, if tile type
$i$ occurs with frequency $f_i$, and if $\phi^n(t_j)$ is an 
$n$-th order supertile of volume $V$, then 
\begin{equation}\label{estimate01} 
|\hbox{Number of tiles of type $i$ in $\phi^n(t_j)$}-
  (f_i \times \hbox{Volume of $\phi^n(t_j)$})| \le c_0 |\lambda_2|^n n^{\nu_0},
\end{equation} 
where $\lambda_2$ is the second-largest eigenvalue of $M$.
\end{lemma}

\begin{proof} This is straightforward linear algebra. If $v_n$ is a column
vector whose $i$-th entry is the number of tiles of type $i$ in $\phi^n(t_j)$
minus $f_i$ times the volume of $\phi^n(t_j)$, then $v_{n+1}=Mv_n$ and
$v_n$ has no component
in the Perron-Frobenius eigenspace of $M$. The vector $v_n$ thus grows at
most as $|\lambda_2|^n$ if $M$ is diagonalizable, and at most 
as a polynomial in $n$ times $|\lambda_2|^n$ if $M$ is not diagonalizable. 
\end{proof}

\begin{lemma}\label{balls}
There exist constants $c$ and $\nu$ such that, for any ball $R$ of 
volume $V$,
\begin{equation}\label{estimate1} 
|\hbox{Number of tiles of type $i$ in $R$}-f_i
  V| < c (\log V)^{\nu} V^{\max\left (\frac{d-1}{d}, \frac{\log|\lambda_2|}{
\log \lambda_1}\right )}. 
\end{equation} 
\end{lemma}

\begin{proof} Let $n$ be the smallest integer such that $V$ is less
  than the volume of the smallest $n+1$-st order supertile. This means
  that $V$ is bounded both above and below by a constant times
  $\lambda_0^{nd}$.  We write $R$ as the union of supertiles of order
  $m$, where $m$ ranges from $0$ to $n$, plus a number of partial
  tiles at the boundary of $R$. First we identify any complete $n$-th
  order supertiles inside $R$, then identify the complete $(n-1)$-st
  order supertiles in the remainder of $R$, then identify the complete
  $(n-2)$-nd order supertiles in what is left, and so on.  The number
  of partial tiles is bounded by a constant times the surface area of
  $R$, which goes as $V^{\frac{d-1}{d}}\sim \lambda_0^{n(d-1)}$. The
  number of supertiles of order $m$ is bounded by a constant times the
  surface area of $R$ scaled down by $\lambda_0^m$, hence
  $\lambda_0^{(n-m)(d-1)}$. The contributions of each supertile are
  governed by equation (\ref{estimate01}), so the total contribution
  of all the supertiles of level $m$ is of order
  $\lambda_0^{n(d-1)}\left ( \frac{|\lambda_2|}{\lambda_0^{d-1}}
  \right )^m m^{\nu_0}$.

The sum
$\sum_{m=0}^n \lambda_0^{n(d-1)} \left ( 
\frac{|\lambda_2|}{\lambda_0^{d-1}} \right )^m m^{\nu_0}$ is bounded by
the number of terms times the largest term, hence by a constant times 
$n^{\nu_0+1}$ times either $|\lambda_2|^n$ or $\lambda_0^{n(d-1)}$,
whichever is larger.
Since $V$ is of order $\lambda_0^{nd}$, $|\lambda_2|^n$ is of order 
$V^{\frac{\log|\lambda_2|}{\log(\lambda_1)}}$, while $\lambda_0^{n(d-1)}$
is of order $V^{\frac{d-1}{d}}$. 
\end{proof} 
 
\begin{proof}[Proof of Theorem \ref{sub-rate}]

  First suppose that the substitution forces the border, so that
  $\check H^d(\Omega, \Q)$ is generated by the duals to the various
  (uncollared) tile types \cite{AP}.  This implies that the top
  pattern-equivariant cohomology is generated by the indicator
  cochains of these tile types, and that we can take our patches $P_1,
  \ldots, P_k$ to simply be different types of uncollared tiles.  By
  Theorem \ref{ergodic}, the number of appearances of $P$ converges no
  slower than the slowest of the $P_i$'s. However, the number of each
  tile type in the region $R$ is governed by equation
  (\ref{estimate1}), which is tantamount to Theorem \ref{sub-rate}.

  If the substitution does not force the border, then we rewrite it
  using collared tiles. This changes the substitution matrix, but
  equation (\ref{estimate1}) still applies, albeit with $\lambda_2$
  being the second-largest eigenvalue of the new matrix. If this is
  the same as the second-largest eigenvalue of the old matrix, then we
  are done. If not, then $|\lambda_2| \le \lambda_0^{d-1}=
  \lambda_1^{\frac{d-1}{d}}$, so $\max\left (\frac{d-1}{d},
    \frac{\log|\lambda_2|}{ \log (\lambda_1)}\right ) =
  \frac{d-1}{d}$.  Dividing by $V$, we again get the estimate of
  Theorem \ref{sub-rate}.
\end{proof}

\begin{proof}[Proof of Theorem \ref{cylinder}]

Without loss of generality, we assume that the substitution $\phi$ is proper. 
Imagine applying the indicator cochain $\chi_P$ to $\phi^n(\ell_2)$ 
for some letter
$\ell_2$ that sits in the 3-letter word $\ell_1\ell_2\ell_3$. 
Since the beginning of 
$\phi(\ell_2)$ and $\phi(\ell_3)$ are the same, and since the end of 
$\phi(\ell_1)$ and $\phi(\ell_2)$ are the same, the exact piece 
$\delta\alpha$ in the
decomposition (\ref{chiP}) of $\chi_P$ evaluates to zero on 
$\phi^n(\ell_2)$ if $n$ is sufficiently large. 
This means that counting $P$'s in $\phi^n(\ell_2)$, for $n$ large, 
is purely a cohomological calculation.  

Now let $\beta$ be a rational pattern-equivariant cochain. We say that
$\beta$ is {\em regular} if, for any letter $\ell$ and for
any sufficiently large $n$, $\beta(\phi^n(\ell))$ is an integer. Clearly,
every indicator cochain is regular, but a rational linear combination
of indicator cochains may not be.

Let $\gamma$ be a $1$-cochain, which we can write as a
linear combination $\sum_{P} c_P \chi_P$ of indicator cochains with
rational coefficients. Define the {\em trace of $\gamma$} to be 
\begin{equation}
Tr(\gamma) = \sum_P c_P f_P,
\end{equation}
where $f_P$ is the frequency of the patch $P$. This trace is the 
same as the Ruelle-Sullivan map in the context of \cite{KP}, and is 
closely related to the trace operation on $K^0$ of the $C^*$-algebra
defined by the action of translations on $\Omega_T$.

\begin{lemma}\label{lem1} 
Let $\bT$ be a 1-dimensional tiling obtained from a primitive substitution
$\phi$.
Let $p(x)$ be the minimal polynomial of the operator $A$ that represents the
action of substitution on $H^1_{PE}(\bT,\Q)$, let $q(x)$ be the minimal 
polynomial of the Perron-Frobenius eigenvalue of the substitution matrix,
and let $p(x)=q(x)r(x)$. Let $L$ be a return length. If $\beta$ is a regular 
$1$-cochain and $q(A)[\beta]=0$, then 
\begin{equation}
Tr(\beta) = \frac{u_\beta(\lambda)}{L q'(\lambda) q_0^n}
\end{equation}
for some polynomial $u_\beta$ with integer coefficients.
\end{lemma}

\begin{proof}
The following proof is a small modification of the proof of Theorem 9 of 
\cite{HPC}. Let $s$ be the algebraic degree of $\lambda$ and let
$w$ be a return word of length $L$. 
Since $\phi(w), \phi^2(w), \ldots$ are return words, 
$\lambda L$, $\lambda^2 L, \ldots$
are return lengths, and $L\Z[\lambda]$ is a rank-$s$ Abelian group, hence
a finite-index subgroup
of the span of all the tile lengths. In particular, the length of any tile
$t$ can be written uniquely 
in the form $|t| = L(c_0(t) + c_1(t) \lambda + \ldots + 
c_{s-1}(t) \lambda^{s-1})$, where each $c_i(t)$ is rational. 
Define pattern-equivariant $1$-cochains $\xi_0,\ldots\xi_{s-1}$ 
by $\xi_i(t) = c_i(t)$.
It is not hard to see that
the cohomology classes $[\xi_i]$ are linearly independent (see Lemma 8 of
\cite{HPC}). Under substitution, the cochains
$\xi_0, \ldots, \xi_{s-1}$ transform via the matrix
\begin{equation}
  C = \begin{pmatrix} 0 & 0 & \cdots & 0 & -q_0 \cr
    1 & 0 & \cdots & 0 & - q_1 \cr
    0 & 1 & \cdots & 0 & -q_2 \cr
    \vdots &  & \ddots & \vdots & \vdots \cr
    0 & 0 & \cdots & 1 & -q_{s-1}
\end{pmatrix},
\end{equation}
where $q(x)=x^s + q_{s-1}x^{s-1} + \cdots + q_0$ is the 
minimal polynomial of $\lambda$. Note that the characteristic polynomial of
$C$ is precisely $q(x)$, and hence that the classes 
$[\xi_i]$ span the kernel of $q(A)$. In particular, the cohomology class 
$[\beta]$ is a rational linear combination of the $[\xi_i]$'s. We can
therefore write $\beta = \sum \beta_i \xi_i + \delta \alpha$, where $\alpha$
is a pattern-equivariant 0-cochain and the coefficients $\beta_i$ are 
rational. Applied to any word of the form
$\phi^n(w)$, for $n$ large enough, $\delta \alpha$ yields zero, while
$\beta(\phi^n(w))$ yields an integer. 

Note that, for any integers $n \ge i \ge 0$, 
$q_0^{n-i}\lambda^i$ is a linear combination of $\lambda^n$, 
$\lambda^{n+1}, \ldots, \lambda^{n+s-1}$.
To see this, divide the equation $q(\lambda)=0$ by $\lambda$ to get 
$q_0 \lambda^{-1} = -(\lambda^{s-1} + q_{s-1}\lambda^{s-2} + \cdots + q_1)$.
Taking the ${n-i}$-th power, applying the equation $q(\lambda)=0$ 
to eliminate large powers of $\lambda$, and finally 
multiplying by $\lambda^n$, gives the result. 
This implies that $q_0^{n-i}(\beta - \delta \alpha)$, applied to $\phi^i(w)$, 
yields an integer, which in turn implies that each 
coefficient $\beta_i$ is an integer divided by $q_0^i$.
 
The trace of $\xi_i$ is easily computed. We just apply $\xi$ to $\phi^n(w)$,
divide by the length of $\phi^n(w)$, and take the limit as $n \to \infty$. 
This is equivalent to writing $\lambda^n = \sum_{i=0}^{s-1} c_{n,i}\lambda^i$
and taking the limit of $c_{n,i}/\lambda^n$, and is precisely the $i$-th
entry of the right eigenvector $\vec v_r$ of $C$, with eigenvalue $\lambda$,
normalized so that 
$(1,\lambda,\ldots,\lambda^{s-1})\vec v_r = 1$. This eigenvector is:
\begin{equation}
  \vec v_r = \frac{1}{q'(\lambda)} \begin{pmatrix}
    \lambda^{s-1} + q_{s-1} \lambda^{s-2} + \cdots + q_1  \cr
    \lambda^{s-2} + q_{s-1} \lambda^{s-3} + \cdots + q_2 \cr
    \lambda^{s-3} + q_{s-1} \lambda^{s-4} + \cdots + q_3 \cr
    \vdots \cr
    \lambda + q_{s-1} \cr
    1
\end{pmatrix}.
\end{equation}
Since each entry of $\vec r$ is a polynomial in $\lambda$ divided by
$q'(\lambda)$, and since each $\beta_i$ is an integer divided by a power of
$q_0$, the trace of $\beta$ is of the desired form. 
\end{proof}

\begin{lemma}\label{lem2}
Let $\beta$ be any pattern-equivariant 1-cochain. 
If $r(A)[\beta]=0$, then $Tr(\beta)=0$. 
\end{lemma}

\begin{proof}
If $r(A)[\beta]=0$, then the cohomology class $[\beta]$ has no component
that scales under $n$-fold substitution as $\lambda^n$. Thus
$\lim \beta(\phi^n(\ell))/\lambda^n = 
\lim (A^n(\beta))(\ell)/\lambda^n = 0$.
Since, in the limit, $\beta$ averages to zero on patches of the form
$\phi^n(\ell)$, the trace of $\beta$ is zero.
\end{proof}

Finally, let $\chi_P$ be an indicator cochain. Since $\chi_P$ is regular, 
$A(\chi_P)
= \chi_P \circ \phi$ is regular, and so is any polynomial in $A$ applied 
to $\chi_P$. In particular, $Q(A)q(A)\chi_P$
and $R(A)r(A)\chi_P$ are regular. Since $[r(A)Q(A)q(A)\chi_P] = p(A)[Q(A)\chi_P]
=0$, the trace of $R(A)r(A)\chi_P$ is of the form indicated 
in Lemma \ref{lem1}. Likewise, the trace of $Q(A)q(A)\chi_P$ is zero by
Lemma \ref{lem2}.  Since $D\chi_P = Q(A)q(A) \chi_P + R(A)r(A) \chi_P$, the
trace of $\chi_P$ is $D^{-1}$ times something of the form indicated in 
Lemma \ref{lem1}, which completes the proof of 
Theorem \ref{cylinder}. 
\end{proof} 

\section{Examples}\label{examples}

\subsection{The Thue-Morse Tiling} 

The Thue-Morse tiling is a 1-dimensional tiling given by the
substitution $\phi_1(a)= ab$, $\phi_1(b)=ba$. 
This substitution is not proper, but being proper is not a requirement
for Theorems \ref{general} and \ref{cylinder}.
In the case of Thue-Morse, it is not difficult to compute the cohomology
directly by a variety of methods. The first rational PE-cohomology is
$\Q^2$, and is generated by the indicators of the patches $P_1=ab$ and $P_2=aa$.
The matrix $A$ (in an appropriate basis) is $\begin{pmatrix} 1&1 \cr 2&0
\end{pmatrix}$ with eigenvalues $\lambda_{PF}=2$ and $\lambda_2=-1$.  
Note that $q(x)=x-2$, $r(x)=x+1$, and $D=3$, insofar as $3=(x+1)-(x-2)$. We
will show directly that the appearance of every patch in a return word is 
governed, up to coboundaries, by the appearance of $P_1$ and $P_2$. Note that
the frequencies of $P_1$ and $P_2$ are $1/3$ and $1/6$, respectively, which 
are {\em not} in $Z[1/2]$. The factor of $D^{-1}$ in Theorem \ref{cylinder} is  
indeed necessary. 

Let $P_3 = aababb$. We will show how the appearance of $P_3$ is controlled by
the appearance of $P_1$ and $P_2$. For definiteness, pick indicator cochains
$\chi_i$ ($i=1$, 2, 3) that equal one on the first letter of $P_i$ and are zero
elsewhere.  We will show that, for any region $R$, $\chi_3(R) = 
\frac{7}{8} \chi_2(R) - \frac{1}{8} \chi_1(R) +$ boundary terms.  

We begin by evaluating each $\chi_i$ on $\phi_1^3(a)=abbabaab$ and 
$\phi_1^3(b)=baababba$. The results are:
\begin{eqnarray}
\chi_1(\phi_1^3(a)) &=& 3 \qquad  \chi_1(\phi_1^3(b))= \begin{cases}
2 & \hbox{if $\phi_1^3(b)$ is followed by $\phi_1^3(a)$} \\
3 & \hbox{if $\phi_1^3(b)$ is followed by $\phi_1^3(b)$} 
\end{cases} \cr
\chi_2(\phi_1^3(a)) &=& 1 \qquad  \chi_2(\phi_1^3(b))= \begin{cases}
2 & \hbox{if $\phi_1^3(b)$ is followed by $\phi_1^3(a)$} \\
1 & \hbox{if $\phi_1^3(b)$ is followed by $\phi_1^3(b)$} 
\end{cases} \cr
\chi_3(\phi_1^3(a)) &=& 
\begin{cases}
0 & \hbox{if $\phi_1^3(a)$ is followed by $\phi_1^3(a)$} \\
1 & \hbox{if $\phi_1^3(a)$ is followed by $\phi_1^3(b)$} 
\end{cases}
 \qquad  \chi_3(\phi_1^3(b))= 1
\end{eqnarray}
Next, suppose that $w$ is a return word. The number of 
times that the patch $ab$ appears in $\phi_1(w)$ equals the number 
of times that $ba$ appears. Also, $\phi_1(w)$ has as many appearances of $a$ as
of $b$. Thus, if we treat $\phi_1^4(w)$ as
the concatenation of 3rd order ``supertiles'' of the form 
$\phi_1^3(a)$ and $\phi_1^3(b)$, then there are equal numbers of $a$ and
$b$ type supertiles, and the number of $a$-supertiles that are followed
by $b$-supertiles equals the number of $b$-supertiles that are followed
by $a$-supertiles. If there are $k_1$ $a$-supertiles, $k_2$ of which are
followed by $b$-supertiles, then
\begin{eqnarray}
\chi_1(\phi_1^4(w)) &=& 6k_1 - k_2; \cr 
\chi_2(\phi_1^4(w)) &=& 2 k_1 + k_2 \cr
\chi_3(\phi_1^4(w)) &=& k_1 + k_2 = -\frac{1}{8}\chi_1(\phi_1^4(w))
+ \frac{7}{8} \chi_2(\phi_1^4(w)).
\end{eqnarray}

Finally, let $R$ be any region. We can always write $R=p\phi_1^4(w)s$, where
$w$ is a return word and the prefix $p$ and the suffix $s$ each have length
at most 48. The number $n_i$ of occurrences of $P_i$ in the prefix and suffix
do {\em not} have to satisfy $n_3 = \frac{-n_1}{8} + \frac{7n_2}{8}$, but
the deviation from this rule is computable from the local patterns $p$ and 
$s$. We therefore have a pattern-equivariant 0-cochain $\alpha$, with
radius at most 48, such that 
\begin{equation}
\chi_3 = \frac{-\chi_1}{8} + \frac{7 \chi_2}{8} + \delta \alpha.
\end{equation}  

The exact same argument would work for any patch $P_4$ of length at most 8. 
We just have to evaluate $\chi_4$ on $\phi_1^3(a)$ and $\phi_1^3(b)$, and use
this information to evaluate $\chi_4$ on $\phi_1^4(w)$ for any return word $w$,
with the answer being a linear function of $k_1$ and $k_2$. 
Since $\chi_1(\phi_1^4(w))$ and $\chi_2(\phi_1^4(w))$ are linearly independent
functions of $k_1$ and $k_2$, we can always express $\chi_4(\phi_1^4(w))$
as a linear combination of $\chi_1(\phi_1^4(w))$ and $\chi_2(\phi_1^4(w))$. 

If we have a patch $P_5$ that is longer than 8 letters, we just have to 
work with higher-order supertiles. If $2^{n-1} < |P_5| \le 2^n$, we count
the appearances of $P_1$, $P_2$ and $P_5$ on $\phi_1^n(a)$ and $\phi_1^n(b)$,
and write an arbitrary word as $p \phi_1^{n+1}(w) s$, where $w$ is a return word
and $p$ and $s$ are words of length at most $6 \cdot 2^n$. 
 
\subsection{Thue-Morse variants}

A couple of variants on the Thue-Morse substitution help illustrate
the extent to which the bounds of Theorem \ref{cylinder} are sharp. In
both cases, as with the original Thue-Morse substitution, the
stretching factor is a power of 2, both tiles can be given length 1,
and there are return words of length 1, 
so Theorem \ref{cylinder} essentially says that all patch
frequencies live in $\frac{1}{D} \Z[1/2]$. 

The first variant is the substitution $\phi_2=\phi_1^4$, or explicitly
$\phi_2(a) =abbabaabbaababba$, 
$\phi_2(b)=baababbaabbabaab$. 
The tilings for $\phi_2$ are exactly the same as those for $\phi_1$. 
In particular, the
appearance of the patches $ab$ and $aa$ govern the appearance of all patches,
and all patch frequencies live in $\frac{1}{3}\Z[1/2]$. 

However, the substitution, acting on cohomology, has eigenvalues
16 and 1 rather than 2 and $-1$, and the number theoretic constant $D$ 
is now 15 rather than 3. 
A naive application of Theorem \ref{cylinder} says
that all patch frequencies live in $\frac{1}{15}\Z[1/2]$, which is true,
but this is not 
sharp; the factor of 5 in the denominator is spurious. 

Now consider the substitution $\phi_3(a)=aaaaaaaabbbbbbbb=a^8b^8$, 
$\phi_3(b) = a^7b^9$. This tiling space also has $H^1=\Q^2$, and the 
substitution, acting on cohomology, has eigenvalues 16 and 1. 
However, in this example there {\em are} patches whose frequencies are not
in $\frac{1}{3}\Z[1/2]$. In particular, the letter $a$ occurs with frequency
$\frac{7}{15}$, while the letter $b$ occurs with frequency $\frac{8}{15}$. 

Theorem \ref{cylinder} appears to be the strongest estimate that can
be made using only the minimal polynomial of $A$, or equivalently using the
eigenvalues of $A$ (and the size of any Jordan blocks).  However, as
$\phi_2$ shows, we can sometimes obtain stronger estimates by studying
the eigen{\em vectors\/} of $A$. The key is decomposing the cohomology class of 
an arbitrary indicator cochain into the sum of two pieces, one that is 
annihilated by $q(A)$ and one that is annihilated by $r(A)$. Although we 
can always do this using rational coefficients with denominator $D$, 
some matrices $A$ allow us to do better. 

\subsection{A Fibonacci variant}

Theorems \ref{general}--\ref{cylinder} 
were stated in terms of the rational cohomology of the tiling
space. This avoids complications relating to torsion and divisibility. 
However, there are times when integer-valued cohomology can be used 
more effectively.

Consider the 1-dimensional substitution on two letters $\phi(a)=baaab$, 
$\phi(b)=aba$. This is an irreducible Pisot substitution, with 
substitution matrix $A = \begin{pmatrix} 3 & 2 \cr 2 & 1 \end{pmatrix}$
and stretching factor $\lambda = 2 + \sqrt{5}$, which is the cube of 
the golden mean. The first cohomology is $\check H^1(\Omega_{\phi})=\Z^3$,
with generators corresponding to the indicator cochains of $P_1=a$, $P_2=b$ 
and $P_3=ab$.

Since every indicator cochain is cohomologous to an integer linear combination
of $\chi_1$, $\chi_2$ and $\chi_3$, every patch frequency is an integer 
linear combination of $f_1 = \frac{1}{2 \sqrt{5}}$, 
$f_2 = \frac{\sqrt{5}-1}{4\sqrt{5}}$, and $f_3=\frac{1}{4 \sqrt{5}}$, where
we have chosen the tiles to have length $|a|=\sqrt{5}+1 = \lambda-1$ and
$|b|=2$. In other words, all patch frequencies are of the form
$\frac{(m + n \sqrt{5})}{4 \sqrt{5}}$, where $m$ and $n$ are integers.

This is stronger than applying Theorem \ref{cylinder} 
with the return length $L=|b|=2$, which only says
that patch frequencies must be of the form 
$\frac{(m + n \sqrt{5})}{16 \sqrt{5}}$.

\subsection{A random tiling}

Next consider a random 1-dimensional tiling $\bT$, with two tile
types, each of length 1. We assume that the label of the tiles are
chosen independently, with each tile having a probability $p$ of being
type $a$ and a $1-p$ probability of being type $b$. With probability
one, every finite word in $a$ and $b$ appears in $\bT$, with a
well-defined overall frequency given by the Bernoulli measure. If $p$
is transcendental, then the frequencies of $a$, $aa$, $aaa$, etc. are
all linearly independent over the rationals. This implies that the
pattern-equivariant cohomology of $\bT$ is infinitely generated.
\footnote{Strictly speaking this is not a consequence of Theorems
  \ref{general} and \ref{ergodic}, since $\bT$ is not
  repetitive. However, $\bT$ being in the support of the Bernoulli
  measure is an adequate substitute for repetitivity and unique
  ergodicity.}  
If $p$ is algebraic, or even rational, then the PE
cohomology is {\em still} infinitely generated, since the set of
possible patches is independent of $p$ as long as $0<p<1$.

\subsection{The equithirds tiling}

We have limited ourselves to one-dimensional examples so far because, 
in one dimension, it is possible to eliminate the error term in 
Theorem \ref{general} by choosing an appropriate return word. In two dimensions,
that is usually much more difficult. However, some two-dimensional tilings, 
like
the half-hex, admit regions whose boundaries are homologically trivial.
Another such example is the equithirds tiling, discovered independently
by Ludwig Danzer (unpublished) and Bill Kalahurka \cite{bill}. 

\begin{figure}
\includegraphics[height=1.8truein]{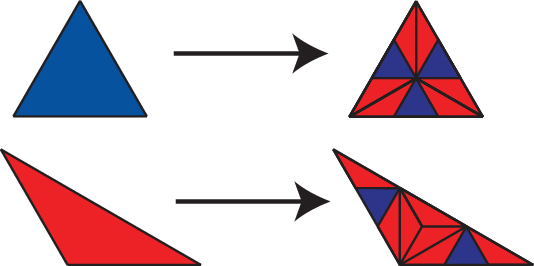}
\caption{The Equithirds Substitution}
\label{fig1}
\end{figure}

The equithirds tiling is a two-dimensional substitution tiling based on the
substitution of Figure \ref{fig1}.  
Each tile is either an equilateral
triangle of side length 1, or a 30-30-120 triangle with sides of length
1, 1, and $\sqrt{3}$. The equilateral triangle appears in two orientations,
while the isosceles triangle appears in six orientations. All triangles
have area $\sqrt{3}/4$. The 
isosceles triangles come in pairs, forming rhombi, and the equilateral
triangles also come in pairs, also forming rhombi. 

The set of vertices of an equithirds tiling is a translate of the
triangular lattice $L$ generated by $(1,0)$ and $(\frac{1}{2},
\frac{\sqrt{3}}{2})$. The unit cell of this lattice has area 
$V_0=\frac{\sqrt{3}}{2}$, or twice the area of a triangle. 
The vertices of first-order supertiles comprise
a translate of $3L$, and vertices of $n$-th order supertiles comprise
a translate of $3^nL$. The locations of $n$-th order vertices mod
$3^nL$ gives a map from the tiling space to the 2-torus, and the
collection of all such locations gives a map from the tiling space to
$\Lim (L, \times 3)$, which is topologically the product of two
3-adic solenoids. This map is a measurable conjugacy, and one might
expect all patch frequencies to live in
$\frac{1}{V_0}\Z[1/9]$.\footnote{We write $\Z[1/9]$ rather than
  $\Z[1/3]$ to emphasize that substitution corresponds to
  multiplication by $9$, but of course the set of 9-adic rational numbers
is the same as the set of 3-adic rationals.}

This is not the case. Each orientation of equilateral triangle, and each
orientation of isosceles triangle, actually appears with frequency 
$\frac{1}{4V_0}$. Overall, 3/4 of the triangles are isosceles, while 1/4 are
equilateral. The appearance of a factor of 4 is analogous to the 
appearance of 1/3 in the patch frequencies of the Thue-Morse tiling.

\begin{figure}
\includegraphics[height=1.8truein]{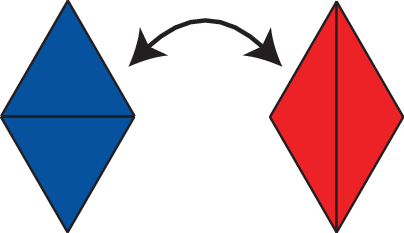}
\caption{Triangles Assemble into Rhombi of Two Types}
\label{fig2}
\end{figure}

\begin{figure}
\includegraphics[height=1.8truein]{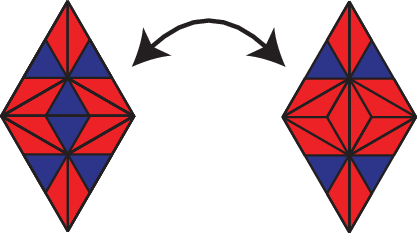}
\caption{Substituted Rhombi}
\label{fig3}
\end{figure}

Moreover, the patches shown in Figure \ref{fig2}, and these patches
rotated by multiples of 120 degrees, play a role analogous to return
words. When substituted one or more times, the patterns on opposite
legs of the rhombus match perfectly, as seen in Figure \ref{fig3}.
The term $\delta \alpha$ in equation (\ref{chiP}) vanishes when
applied to a sufficiently substituted rhombus. This means that the
number of appearance of any patch $P$ in a sufficiently substituted
rhombus is determined exactly by the number of appearances of the
control patches $P_1, \ldots, P_k$.
 
The equithirds tiling forces the border, which makes it easy to
compute the cohomology \cite{bill} using the methods of Anderson and Putnam
\cite{AP}. The answers are that $\check H^1(\Omega) = \Z[1/3]^2$ and
$\check H^2(\Omega) = \Z[1/9] \oplus \Z^3$, hence $\check H^2(\Omega,
\Q)= \Q^4$, or $k=4$.  We can choose our control patches as follows.
Let $P_1$ be an equilateral triangle with horizontal base, let $P_2$
be the second rhombus shown in Figure \ref{fig2}, and let $P_3$ and
$P_4$ be rotated versions of $P_2$. The frequency of each $P_i$ is
$\frac{1}{4V_0}$, and every patch frequency lives in
$\frac{1}{4V_0}\Z[1/9]$.

Finally, it is possible to directly understand the various terms in
$\check H^2(\Omega)$. The patches shown in Figure \ref{fig3}, obtained
by substituting the rhombi of Figure \ref{fig3}, are identical
except in the very middle. Substituting again yields
even bigger patches that are identical except for one rhombus in the
middle. Taking a limit we obtain two tilings that agree completely
except for that one rhombus. One of the classes in $\check
H^2(\Omega)$, represented by $\chi_1-\chi_2$, measures the difference
between these two tilings, and is invariant under substitution. 
Likewise, $\chi_1-\chi_3$ and $\chi_1 -
\chi_4$ measure the same thing rotated by 120 and 240
degrees. Each of these classes has trace zero. 
A fourth class, $\chi_1+\chi_2+\chi_3+\chi_4$, counts area in
units of $V_0$, scales by 9 under substitution, and has trace 
$\frac{1}{V_0}$. This corresponds to 
a generator of $\Z[1/9]$. These four classes span $\check H^2(\Omega, \Q)$, 
but some indicator cochains cannot be written as integer linear combinations.
For instance, $\chi_1 = \frac{1}{4}\left [ (\chi_1+\chi_2+\chi_3+\chi_4) +
(\chi_1-\chi_2) + (\chi_1-\chi_3) + (\chi_1-\chi_4) \right ]$, so the
trace of $\chi_1$ (and likewise $\chi_{2,3,4}$) is $\frac{1}{4V_0}$.

\subsection{The pinwheel tiling}

As noted earlier, the pinwheel tiling \cite{pinwheel} falls outside the
assumptions of this paper. It is not translationally finite, and each 
patch of a tiling appears with frequency zero. To make sense of patch
frequencies, we must count how many times a patch {\em or a ratated 
version of that patch} appears per unit area. This involves looking at
rotationally invariant indicator cochains. (These are a special
case of the pattern-equivariant 
cochains with a representation developed in \cite{BG}.)
Since the coboundary
of a rotationally invariant cochain is rotationally invariant, we can 
define a rotationally invariant pattern-equivariant cohomology 
$H_{PE,rot}^n(\bT,\Q)$ to be the closed rotationally invariant cochains 
modulo coboundaries of rotationally invariant cochains. 

The correspondence between $H_{PE,rot}^n(\bT,\Q)$ and the \v Cech cohomology
of the pinwheel tiling space is subtle.
$H^2_{PE,rot}(\bT,\Q)$ actually corresponds to $\check H^3(\Omega, \Q)$,
where $\Omega$ is a compactification of the hull of a pinwheel tiling, using
a metric where two tilings are close if they agree on a big ball up to a 
small rigid motion (which may include a small rotation). See \cite{BDHS}
for this correspondence and for a computation of $\check H^*(\Omega, \Q)$, 
with the result that $\check H^3(\Omega, \Q)=\Q^8$, hence that $H_{PE, rot}^2
(\bT, \Q)=\Q^8$, hence that there are eight patches 
that control the appearance of all other patches, up to 
boundary terms. 

\section*{Acknowledgments}

I thank Marcy Barge, Henk Bruin, Natalie Frank, Leslie Jones, Bill
Kalahurka and the
participants in the 2010 CIRM Workshop on Subshifts and Tilings for
helpful discussions. I also thank Bill Kalahurka for figures of the
equithirds tiling, and the anonymous referee for pointing out an
error in an earlier version of Theorem \ref{sub-rate}. 
This work is partially supported by the National
Science Foundation.


\bigskip

Department of Mathematics, University of Texas, Austin, TX 78712, USA \\
sadun@math.utexas.edu


\begin{thebibliography}{BDHS}
\bibitem[A]{Aus} J.\ Auslander, Minimal flows and their extensions,
  {\em North-Holland Mathematical Studies}, vol. 153, North-Holland,
  Amsterdam, New York, Oxford, and Tokyo, (1988).

\bibitem[AP]{AP} J.\ E.\ Anderson and I.\ F.\ Putnam, Topological invariants
  for substitution tilings and their associated $C^*$-algebras, {\em
    Ergodic Theory \& Dynamical Systems} \textbf{18} (1998), 509--537.




\bibitem[BBG]{BBG} J.~Bellissard, R.~Benedetti, and J.-M.~Gambaudo, 
Spaces of tilings, finite telescopic approximations
and gap-labelling, Comm. Math. Phys. {\bf 261} (2006), 1-–41.

\bibitem[BBJS]{HPC} M.~Barge, H.~Bruin, L.~Jones and L.~Sadun, 
Homological Pisot Substitutions and Exact Regularity. Available as
arXiv:1001.2027v1.


\bibitem[BD]{BD} M.\ Barge and B.\ Diamond, Cohomology in
  one-dimensional substitution tiling spaces, {\em
    Proc. Amer. Math. Soc.} {\bf 136} (6) (2008), 2183-2191.

\bibitem[BDHS]{BDHS} M.~Barge, B.~Diamond, J.~Hunton and L.~Sadun, 
Cohomology of Substitution Tiling Spaces, to appear in 
Ergodic Theory and Dynamical Systems. Available at arXiv:0811.2507.

\bibitem[CGU]{Jean-Marc} J.-R.~Chazottes, J.-M.~Gambaudo, and E. Ugalde,
On the Geometry of Ground States and Quasicrystals in Lattice Systems,
preprint {\tt arXiv:0802.3661}.

\bibitem[CS]{CS} A.~Clark and L.~Sadun, When Shape Matters: 
Deformations of Tiling Spaces, {\em Ergodic Theory and Dynamical Systems} 
{\bf 26} (2006) 69--86. 








\bibitem[G]{Gahler} F. Gaehler, Lectures given at workshops {\em
    Applications of Topology to Physics and Biology},
  Max-Planck-Institut fr Physik komplexer Systeme, Dresden, June
  2002, and {\em Aperiodic Order, Dynamical Systems, Operator Algebras
    and Topology}, Victoria, British Columbia, August, 2002.






\bibitem[Kal]{bill} W.~Kalahurka, 
Rotational Cohomology and Total Pattern Equivariant
Cohomology of Tiling Spaces Acted on by Infinite
Groups, PhD thesis in Mathematics, University of Texas, 2010. 

\bibitem[Kel1]{K1} J.~Kellendonk, Non-commutative Geometry of Tilings and
Gap-Labeling, {\em Rev. Math. Phys.} {\bf 7} (1995) 1133--1180.

\bibitem[Kel2]{Kel} J.\ Kellendonk,
Pattern-equivariant functions and cohomology,
{\em J. Phys. A.} {\bf 36} (2003), 1--8.

\bibitem[KP]{KP} J.\ Kellendonk and I.\ Putnam, The Ruelle-Sullivan map for
$\R^n$-actions, {\em Math. Ann.} {\bf 334} (2006), 693--711.

\bibitem[Po]{Pohst} M.~Pohst, A Note on Index Divisors, in ``Computational
Number Theory'', A.~ Pehto, M.~Pohst, H.~Williams and H.~Zimmer, ed.  (1991)
de Gruyter, Berlin, 173--182.

\bibitem[Rad]{pinwheel}C.\ Radin, The Pinwheel Tilings of the Plane, 
Annals of Math. {\bf 139} (1994), 661--702.

\bibitem[Ran]{BG} B. Rand, Pattern-Equivariant Cohomology of Tiling 
Spaces with Rotations, PhD thesis in Mathematics, University of 
Texas, 2006. 

\bibitem[Sa1]{inverse} L.\ Sadun,  Tiling Spaces are Inverse Limits. 
{\em J. Math. Phys.} {\bf 44} (2003), 5410--5414.

\bibitem[Sa2]{integer} L.\ Sadun, Pattern-equivariant cohomology with
  integer coefficients. {\em Ergodic Theory \& Dynamical Systems} {\bf
      27} (2007), 1991--1998.




\bibitem[WT]{Weise-Taixes} X. Taixes i Ventosa, G Wiese, Computing
Congruences of Modular Forms and Galois Representations 
Modulo Prime Powers, preprint arxiv:0909.2724v2.  

\end{thebibliography}
\end{document}